\documentclass[11pt]{amsart}
\theoremstyle{plain}
\newtheorem{thm}{Theorem}[section]
\newtheorem{theorem}[thm]{Theorem}

\newtheorem{lemma}[thm]{Lemma}

\theoremstyle{definition}
\newtheorem{remark}[thm]{Remark}

\newtheorem{definition}[thm]{Definition}

\numberwithin{equation}{section}

\title[Monge-Ampere operators via pseudo-isomorphisms:  well-defined cases]{Complex Monge-Ampere operators via pseudo-isomorphisms: the well-defined cases}

\author{Tuyen Trung Truong}

\address{Department of Mathematics, Syracuse University, Syracuse NY 13244} \email{tutruong@syr.edu}

\thanks{}

\begin{document}

\maketitle

\begin{abstract}
Let $X$ and $Y$ be compact K\"ahler manifolds of dimension $3$. A bimeromorphic map $f:X\rightarrow Y$ is pseudo-isomorphic if $f:X-I(f)\rightarrow Y-I(f^{-1})$ is an isomorphism.  

Let $T=T^+-T^-$ be a current on $Y$, where $T^{\pm}$ are positive closed $(1,1)$ currents which are smooth outside a finite number of points. We assume that the following condition is satisfied:

{\bf Condition 1.} For every curve $C$ in $I(f^{-1})$, then in cohomology $\{T\}.\{C\}=0$.

Then, we define a natural push-forward $f_*(\varphi dd^cu\wedge f^*(T))$ for a quasi-psh function $u$ and a smooth function $\varphi$ on $Y$. We show that this pushforward satisfies a Bedford-Taylor's monotone convergence type.

Assume moreover that the following two conditions are satisfied

{\bf Condition 2.} The signed measure $T\wedge T\wedge T$ has no mass on $I(f^{-1})$.

{\bf Condition 3.} For every curve $C$ in $I(f^{-1})$, the measure $T\wedge [C]$ has no Dirac mass.
  
Then, we define a Monge-Ampere operator $MA(f^*(T))=f^*(T)\wedge f^*(T)\wedge f^*(T)$ for $f^*(T)$. We show that this Monge-Ampere operator satisfies several continuous properties, including a Bedford-Taylor's monotone convergence type when $T$ is positive. The measures $MA(f^*(T))$ are in general quite singular. Also, note that it may be not possible to define $f^*(T^{\pm})\wedge f^*(T^{\pm})\wedge f^*(T^{\pm})$. 
\end{abstract}

\section{Introduction}

Let $X$ and $Y$ be compact K\"ahler manifolds of dimension $3$. A bimeromorphic map $f:X\rightarrow Y$ is pseudo-isomorphic if the map $g=f|_{X-I(f)}:X-I(f)\rightarrow Y-I(f^{-1})$ is an isomorphism. Here $I(f)$ and $I(f^{-1})$ are the indeterminate sets of $f$ and $f^{-1}$, both have dimensions at most $1$. (In fact, Bedford-Kim \cite{bedford-kim} showed that if $I(f)$, and hence $I(f^{-1})$, is non-empty then it must be of pure dimension $1$). We let $\Gamma _g\subset (X-I(f))\times (Y-I(f^{-1}))$ be the graph of $g$, and $\Gamma _f=$ the closure of $\Gamma _g$ in $X\times Y$ the graph of $f$. Let $\pi _1,\pi _2:~X\times Y\rightarrow X,Y$ be the natural projections, and occasionally we use the same notations for the restrictions to $\Gamma _g,\Gamma _f$. 

Given a meromorphic map $f:X\rightarrow Y$ and a smooth closed $(1,1)$ form $\theta$ on $Y$, the pullback $f^*(\theta )$ is well-defined as a $(1,1)$ current, however, in general is singular on $I(f)$. To see an explicit example, consider the simple map $J:\mathbb{P}^3\rightarrow \mathbb{P}^3$ given by the formula $J[x_0:x_1:x_2:x_3]$ $=$ $[1/x_0:1/x_1:1/x_2:1/x_3]$. If $u$ is a smooth function then $J^*dd^cu$ will have singularities of the form $1/(x_k^2\overline{x_l}^2)$ near the curves of indeterminacy $x_i=x_j=0$. Therefore, a priori it is not clear whether we can define the Monge-Ampere operator $MA(J^*dd^cu)=J^*dd^cu\wedge J^*dd^cu\wedge J^*dd^cu$ in a reasonable manner.   

In \cite{truong}, we show that if $f:X\rightarrow Y$ is a pseudo-isomorphism in dimension $3$, and $\theta$ is a smooth closed $(1,1)$ form on $Y$ such that in cohomology $$\{\theta\}.\{C\}=0,$$ for every curve $C$ in $I(f^{-1})$, then we have a well-defined Monge-Ampere operator $MA(f^*(\theta ))$. For example, the map $J$ above is not yet pseudo-isomorphic, but if we let $X\rightarrow \mathbb{P}^3$ be the blowup at the $4$ points $e_0=[1:0:0:0]$, $e_1=[0:1:0:0]$, $e_2=[0:0:1:0]$, $e_3=[0:0:0:1]$, then the induced map $J_X$ is a pseudo-isomorphism. Hence, for any smooth function $u$ on $Y=X$, we can define the Monge-Ampere operator $J_X^*dd^cu$ in a reasonable and consistent way, even though as we saw above this $(1,1)$ current is quite singular. Note that if we write $dd^cu=\alpha ^+-\alpha ^-$, where $\alpha ^{\pm}$ are positive closed smooth $(1,1)$ forms, there may be no reasonable and consistent way to define the wedge products $J_X^*(\alpha ^{\pm})\wedge J_X^*(\alpha ^{\pm})\wedge J_X^*(\alpha ^{\pm})$. For such an intersection to be well-defined, we may want to show that $J_X^*(\alpha ^{\pm})$ have locally bounded potentials near $I(J_X)$. However, we have the following result
\begin{lemma}
Let $C$ be an irreducible curve in $I(J_X)$. Then $J_X(C)=D$ is another irreducible curve in $I(J_X)$. If $\omega $ is a positive closed smooth $(1,1)$ form on  $X$ such that $\{\omega\}.\{D\}>0$, then the local potentials of $J_X^*(\omega )$ are unbounded near $C$. 
\label{Lemma0}\end{lemma}
\begin{proof}
That $J_X(C)=D$ is an irreducible curve in $J_X(C)$ can be checked directly (see the last Section in \cite{truong}). Now we prove the claim about the unboundedness of the local potentials of $J_X^*(\omega )$ near $C$.  Assume otherwise. Then by Bedford-Taylor's results \cite{bedford-taylor}, the wedge intersection of currents $J_X^*(\omega )\wedge [C]$ is well-defined as a positive measure on $X$. In particular, in cohomology
\begin{eqnarray*}
0\leq \{J_X^*(\omega )\wedge [C]\}= \{J_X^*(\omega )\}.\{C\}=\{\omega \}.(J_X)_*\{C\}.
\end{eqnarray*}
However, we can check that in cohomology $(J_X)_*\{C\}=-\{D\}$ (see the last section in \cite{truong}). Hence we obtain
\begin{eqnarray*}
\{\omega \}.(J_X)_*\{C\}=-\{\omega \}.\{D\}<0.
\end{eqnarray*}
by assumption. This is a contradiction. 
\end{proof}

The purpose of this short note is to extend the Monge-Ampere operator $MA(f^*(T))=f^*(T)\wedge f^*(T)\wedge f^*(T)$ in \cite{truong} to currents $T$ which can be singular on a finite number of points. The points are allowed to be in $I(f^{-1})$. The main motivation for this is that given a psef cohomology class $\eta \in H^{1,1}(X)$, it may not be able to find a positive closed smooth form $\theta$ in that class, while if we allow a mild singularity there may be a positive closed $(1,1)$ current in the class of $\eta$ with that singularity. Moreover, if we allow more singularity for $T$, then the current $f^*(T)$ may be more singular and hence it makes it more difficult to define $MA(f^*(T))$.

We show that the Monge-Ampere operator so defined satisfies various continuous properties, see in particular Theorems  \ref{Theorem1} and \ref{Lemma1}, Lemmas \ref{Lemma?} and \ref{Lemma4}, and the last subsection of the paper. In the proof of the continuous properties, we will use the following approximation of positive closed smooth $(1,1)$ currents, due to Demailly \cite{demailly1}.

\begin{definition}
Let $Y$ be a compact K\"ahler manifold with a K\"ahler form $\omega _Y$. Let $T=\alpha +dd^cu$ be a positive closed $(1,1)$ current on $Y$, where $\alpha $ is a smooth closed $(1,1)$ form and $u$ is a quasi-psh function. Let $u_j$ be a sequence of smooth quasi-psh functions on $Y$ decreasing to $u$ such that $\alpha +dd^cu_j+\epsilon \omega _Y\geq 0$ for all $j$, here $\epsilon >0$ is a positive constant. Then we say that $\alpha +dd^cu_j$ is a good approximation of $T=\alpha +dd^cu$.
\label{Definition3}\end{definition}

Here we summarize the main results. 

\begin{theorem}
Let $f:X\rightarrow Y$ be a pseudo-isomorphism in dimension $3$. Let $T=T^+-T^-$ be a difference of two positive closed currents $(1,1)$ currents on $Y$, both are smooth outside a finite number of points. These points are allowed to be in $I(f^{-1})$.

Assume that for every curve $C$ in $I(f^{-1})$ we have in cohomology $\{T\}.\{C\}=0$.

We write $f^*(T)=\Omega +dd^cu$, where $\Omega $ is a smooth closed $(1,1)$ form and $u=u^+-u^-$ is a difference of two quasi-psh functions. 

1) (Bedford-Taylor's monotone convergence.) Let $u_j^{\pm}$ be smooth quasi-psh functions decreasing to $u^{\pm}$. Then for any smooth function $\varphi$ on $X$, the following limit exists
\begin{eqnarray*}
\lim _{j\rightarrow\infty}f_*(\varphi (\Omega +dd^cu_j^+-dd^cu_j^-)\wedge f^*(T)).
\end{eqnarray*} 
We denote the limit by $f_*(\varphi f^*(T)\wedge f^*(T))$.

2) Let $S$ be a smooth closed $(1,1)$ form on $Y$. Let $u_j^{\pm}$ be as in 1). Then 
\begin{eqnarray*}
\lim _{j\rightarrow\infty}\int _{X}\varphi f^*(S)\wedge (\Omega +dd^cu_j^+-dd^cu_j^-)\wedge f^*(T)=\int _YS\wedge f_*(\varphi f^*(T)\wedge f^*(T)).
\end{eqnarray*}

3) Assume further that $T$ satisfies the following two conditions:

i) The measure $T\wedge T\wedge T$ has no mass on $I(f^{-1})$.

ii) For each curve $C$ in $I(f^{-1})$ then the measure $T\wedge [C]$ has no Dirac mass. 

Then there is a natural and well-defined wedge intersection of currents $T\wedge f_*(\varphi f^*(T)\wedge f^*(T))$. In view of 2) above, we define the Monge-Ampere operator $MA(f^*(T))$ by the formula
\begin{eqnarray*}
<MA(f^*(T)),\varphi >:=\int _YT\wedge f_*(\varphi f^*(T)\wedge f^*(T)).
\end{eqnarray*}
 
4) Assumptions are as in 3). Assume further that $T$ is a positive current. We write $T=\alpha +dd^cv$, where $\alpha$ is a smooth closed $(1,1)$ form and $v$ is a quasi-psh function. Let $\alpha +dd^cv_n$ be a good approximation of $T=\alpha +dd^cv$, in the sense of Definition \ref{Definition3}. Then for any smooth function $\varphi$ we have
\begin{eqnarray*}
\lim _{n\rightarrow\infty}\int _Y(\alpha +dd^cv_n)\wedge f_*(\varphi f^*(T)\wedge f^*(T))=\int _YT\wedge f_*(\varphi f^*(T)\wedge f^*(T)).
\end{eqnarray*}
 
In other words, we have a double Bedford-Taylor's monotone type convergence 
\begin{eqnarray*}
\lim _{n\rightarrow\infty}\lim _{j\rightarrow\infty}\int _X\varphi f^*(\alpha +dd^cv_n)\wedge (\Omega +dd^cu_j) \wedge f^*(T)=\int _YT\wedge f_*(\varphi f^*(T)\wedge f^*(T)).
\end{eqnarray*}

5) Assumptions are as in 4). Assume moreover that $f^*(T\wedge T)$ has no mass on $I(f)$ (for example if $T$ is smooth near $I(f^{-1})$). If $\Omega +dd^cu_j$ is a good approximation of $f^*(T)=\Omega +dd^cu$ in the sense of Definition \ref{Definition3}, then 
\begin{eqnarray*}
\lim _{j\rightarrow\infty}\int _X\varphi f^*(T)\wedge (\Omega +dd^cu_j)\wedge f^*(T)=\int _YT\wedge f_*(\varphi f^*(T)\wedge f^*(T)).
\end{eqnarray*}

6) Assumptions are as in 3). If $T$ is smooth or positive then $MA(f^*(T))=f^*(T\wedge T\wedge T)$. Here, since the measure $T\wedge T\wedge T$ has no mass on $I(f^{-1})$, the pullback $f^*(T\wedge T\wedge T)$ is well-defined. 
\label{TheoremMain}\end{theorem}

\begin{remark}
In 6) of Theorem \ref{TheoremMain}, a priori the measure $f^*(T\wedge T\wedge T)$ is quite singular near $I(f)$, even if $T$ is smooth. Also, note that there may be no reasonable and consistent manner to define the terms $f^*(T^{\pm})\wedge f^*(T^{\pm})\wedge f^*(T^{\pm})$, so we need to define $f^*(T)\wedge f^*(T)\wedge f^*(T)$ directly. See Lemma \ref{Lemma0} and the discussion before it.
\end{remark}

\section{Definition of the Monge-Ampere operator}

We will consider the following class of currents
\begin{definition}
Class $(\mathcal{A})$. A closed $(1,1)$ current $T$ is in class $(\mathcal{A})$ if $T=T^{+}-T^{-}$ where $T^{\pm}$ are positive closed $(1,1)$ currents which are smooth outside a finite number of points. 
\label{Definition1}\end{definition}

\begin{remark}
The essential property that we need in the above definition is that in $W-A$, here $W$ is an open neighborhood of $I(f^{-1})$ and $A$ is a finite set, the currents $T^{\pm}$ are smooth (in fact, continuous is enough). Outside $W-A$, $T^{\pm}$ may have mild singularity such that $T\wedge T\wedge T$ is well-defined. For example, following Bedford-Taylor \cite{bedford-taylor}, we need only to require that $T^{\pm}$ have locally bounded potentials.
\end{remark}
 
\begin{remark}
That $T^{\pm}$ may have singular points on $I(f^{-1})$ makes it difficult to define the individual wedge products of currents $f^*(T^{\pm})\wedge f^*(T^{\pm})\wedge f^*(T^{\pm})$. This is because the preimage of a point on $I(f^{-1})$ may be a whole curve on $I(f)$. So a priori $f^*(T^{\pm})$ may be singular on a whole curve contained in $I(f)$, see for example Lemma \ref{Lemma0}. Hence, in the below, we will define $f^*(T)\wedge f^*(T)\wedge f^*(T)$ directly, not via the wedge products $f^*(T^{\pm})\wedge f^*(T^{\pm})\wedge f^*(T^{\pm})$. 
\end{remark} 
   
We will consider the following three conditions 

{\bf Condition 1.} For every curve $C$ in $I(f^{-1})$, then in cohomology $\{T\}.\{C\}=0$.

{\bf Condition 2.} The signed measure $T\wedge T\wedge T$ has no mass on $I(f^{-1})$.

{\bf Condition 3.} For every curve $C$ in $I(f^{-1})$, the measure $T\wedge [C]$ has no Dirac mass.

\begin{remark}If $T$ is in Class $(\mathcal{A})$, then the measure $T\wedge T\wedge T$ has no mass on $I(f^{-1})$, except possibly a finite number of points on $I(f^{-1})$ where $T$ is not smooth. Hence Condition 2 is equivalent to that $T\wedge T\wedge T$ has no Dirac masses at these points. 
\end{remark}

\begin{remark}
If $T$ is smooth then $T$ satisfies both Conditions 2 and 3.

If $T$ is a positive closed $(1,1)$ current in Class ($\mathcal{A}$) and satisfies Condition 1, then it automatically satisfies Condition 3. Because in this case the wedge product of currents $T\wedge [C]$ is well-defined as a positive measure, and the total mass is $\{T\}.\{C\}$. However, if $T$ is not positive then this implication is not automatic. 
\end{remark}

Assume that the Monge-Ampere operator $MA(f^*(T))=f^*(T)\wedge f^*(T)\wedge f^*(T)$ is well-defined. Then, formally, for a smooth function $\varphi$ on $X$ we have
\begin{equation}
\int _X\varphi f^*(T)\wedge f^*(T)\wedge f^*(T)=\int _YT\wedge f_*(\varphi f^*(T)\wedge f^*(T)),
\label{Equation0}\end{equation}
provided that both wedge intersections of currents $f_*(\varphi f^*(T)\wedge f^*(T))$ and $T\wedge f_*(\varphi f^*(T)\wedge f^*(T))$ are well-defined. The remaining of this note is to define these under the assumption that $T$ is in Class $\mathcal{A}$ and satisfies Conditions 1, 2 and 3. 

\begin{remark}[Justification for the approach.] Under Condition 1, we showed in \cite{truong} that $f^*(T)\wedge f^*(T)=f^*(T\wedge T)$, so one may attempt to define $MA(f^*(T))$ in a different way
\begin{equation}
\int _X\varphi f^*(T)\wedge f^*(T)\wedge f^*(T)=\int _Xf_*(\varphi f^*(T))\wedge T\wedge T.
\label{Equation?}\end{equation}

At a first look, this approach seems to have equal footing with our approach in Equation (\ref{Equation0}). To justify what approach is more reasonable, let us consider a more general problem. Assume that $S$ is another $(1,1)$ current which is smooth, and we want to define $f^*(S)\wedge f^*(T)\wedge f^*(T)$. 

Our approach in Equation (\ref{Equation0}) is to define
\begin{eqnarray*}
\int _X\varphi f^*(S)\wedge f^*(T)\wedge f^*(T):=\int _YS\wedge f_*(\varphi f^*(T)\wedge f^*(T)).
\end{eqnarray*}
The right hand side of the above expression is well-defined, since $S$ is smooth, provided that the current $f_*(\varphi f^*(T)\wedge f^*(T))$ is defined. Moreover, the equality is justified by proving a continuity property, see Lemma \ref{Lemma?} below. 

The approach in Equation \ref{Equation?} is to define either
\begin{eqnarray*}
\int _X\varphi f^*(S)\wedge f^*(T)\wedge f^*(T):=\int _Yf_*(\varphi f^*(S))\wedge T\wedge T,
\end{eqnarray*} 
or 
\begin{eqnarray*}
\int _X\varphi f^*(S)\wedge f^*(T)\wedge f^*(T):=\int _Yf_*(\varphi f^*(T))\wedge S\wedge T.
\end{eqnarray*}
Since $T$ may not be smooth, the equalities between the two sides of the above two expressions are not justified, if $\varphi$ is not a constant. 
 
From this simple consideration, we see that the definition in Equation (\ref{Equation0}) is more reasonable. Moreover, we will show later that if either $T$ is smooth or positive, then the definitions in Equations (\ref{Equation0}) and (\ref{Equation?}) are the same.
\end{remark}

Now we state and prove the continuous property referred to in the above remark.
\begin{lemma}
(Bedford-Taylor's monotone convergence.) Assume $S$ is a smooth closed $(1,1)$ form and $T$ is a current in the class ($\mathcal{A}$) and satisfies Condition 1. We write $f^*(T)=\Omega +dd^cu$, where $\Omega $ is a smooth closed $(1,1)$ form and $u=u^+-u^-$ is the difference of two quasi-psh functions. Let $u_j^{\pm}$ be a sequence of smooth quasi-psh functions decreasing to $u^{\pm}$. We denote $u_j=u_j^+-u_j^-$. Then 
\begin{eqnarray*}
\lim _{j\rightarrow\infty}\int _X\varphi f^*(S)\wedge (\Omega +dd^cu_j)\wedge f^*(T)=\int _YS\wedge f_*(\varphi (\Omega +dd^cu)\wedge f^*(T)).
\end{eqnarray*}
Here the current $f_*(\varphi (\Omega +dd^cu)\wedge f^*(T))$ is defined in Equation (\ref{Equation1}) below.
\label{Lemma?}\end{lemma}
\begin{proof}
First, we show that for each $j$
\begin{eqnarray*}
\int _X\varphi f^*(S)\wedge (\Omega +dd^cu_j)\wedge f^*(T)=\int _YS\wedge f_*(\varphi (\Omega +dd^cu_j)\wedge f^*(T)).
\end{eqnarray*}
Here both sides are well-defined, since $\varphi$, $S$, $\Omega$ and $u_j$ are smooth. The term $f_*(\varphi (\Omega +dd^cu_j)\wedge f^*(T))$ is defined as follows, by Meo's results:
\begin{eqnarray*}
f_*(\varphi (\Omega +dd^cu_j)\wedge f^*(T))=(\pi _2)_*(\pi _1^*(\varphi (\Omega +dd^cu_j))\wedge \pi _1^*(f^*(T))\wedge [\Gamma _f]).
\end{eqnarray*}

Now we can approximate $f^*(T)$ by smooth closed $(1,1)$ forms $\gamma _n=\gamma _n^+-\gamma _n^-$. Here $\gamma _n^{\pm}$ positive closed smooth $(1,1)$ forms with uniformly bounded masses, and converges locally uniformly on $X-I(f)$ to $f^*(T)$. 

Then it can be seen, by dimension reason (see for example the proof of Lemma 5 in \cite{truong0}), that 
\begin{eqnarray*}
\int _X\varphi f^*(S)\wedge (\Omega +dd^cu_j)\wedge f^*(T)=\lim _{n\rightarrow\infty}\int _Y\varphi f^*(S)\wedge (\Omega +dd^cu_j)\wedge (\gamma _n^+-\gamma _n^-).
\end{eqnarray*}
 
Since all $\varphi$, $S$, $\Omega$, $u_j$ and $\gamma _n^{\pm}$ are all smooth, we have
\begin{eqnarray*}
\int _X\varphi f^*(S)\wedge (\Omega +dd^cu_j)\wedge (\gamma _n^+-\gamma _n^-)=\int _YS\wedge f_*(\varphi (\Omega +dd^cu_j)\wedge (\gamma _n^+-\gamma _n^-)).
\end{eqnarray*}
 
Now 
\begin{eqnarray*}
f_*(\varphi (\Omega +dd^cu_j)\wedge (\gamma _n^+-\gamma _n^-))=(\pi _2)_*(\pi _1^*(\varphi )\pi _1^*(\Omega +dd^cu_j)\wedge \pi _1^*(\gamma _n^+-\gamma _n^-)\wedge [\Gamma _f]).
\end{eqnarray*}

The limit when $n\rightarrow\infty$ of the right hand side is $(\pi _2)_*(\pi _1^*(\varphi )\pi _1^*(\Omega +dd^cu_j)\wedge \pi _1^*(f^*(T))\wedge [\Gamma _f])$. This is because the limit of $\pi _1^*(\varphi )\pi _1^*(\Omega +dd^cu_j)\wedge \pi _1^*(\gamma _n^+-\gamma _n^-)\wedge [\Gamma _f]$ is $\pi _1^*(\varphi )\pi _1^*(\Omega +dd^cu_j)\wedge \pi _1^*(f^*(T))\wedge [\Gamma _f]$. 

Therefore the claim is proved. Using this claim and part 2) of Theorem \ref{Theorem1} below, the lemma follows. 
\end{proof}

\subsection{Definition of the current $f_*(\varphi (\Omega +dd^cu)\wedge f^*(T))$}

We denote by $(\pi _1^*(f^*(T))\wedge [\Gamma _g])^o$ the extension by zero of the current $\pi _1^*(f^*(T))\wedge [\Gamma _g]$ (the latter has  bounded mass by Meo's result \cite{meo}). Let $u$ be a quasi-psh function on $X$. Theorem 1.2 in \cite{truong} shows that the current $\pi _1^*(uf^*(T))\wedge [\Gamma _g]$ has bounded mass, and we let $(\pi _1^*(uf^*(T))\wedge [\Gamma _g])^o$ denote its extension by zero. In \cite{truong}, we defined 
\begin{equation}
f_*(\varphi dd^cu\wedge f^*(T)):=(\pi _2)_*(\pi _1^*(\varphi )\wedge dd^c(\pi _1^*(uf^*(T))\wedge [\Gamma _g])^o)
\label{Equation1}\end{equation}

We now prove a Bedford-Taylor's monotone convergence theorem for this operator.

\begin{theorem}
Assume that $T$ is in Class ($\mathcal{A}$) and satisfies the Condition 1. Then 

1) If $u$ is a smooth quasi-psh function on $X$, we have
\begin{eqnarray*}
f_*(\varphi dd^cu\wedge f^*(T))=(\pi _2)_*(\pi _1^*(\varphi )\wedge dd^c\pi _1^*(uf^*(T))\wedge [\Gamma _f]).
\end{eqnarray*}
The right hand side above is the (correct) usual definition in the case $u$ is smooth. 

2) Let $u$ be a quasi-psh function on $X$, and let $u_j$ be a sequence of smooth quasi-psh functions decreasing to $u$. Then 
\begin{eqnarray*}
\lim _{j\rightarrow\infty}f_*(\varphi dd^cu_j\wedge f^*(T))=f_*(\varphi dd^cu\wedge f^*(T)).
\end{eqnarray*}
\label{Theorem1}\end{theorem}    
\begin{proof}

1) A modification of the proof of Theorem 1.3 in \cite{truong} shows that 
\begin{eqnarray*}
\pi _1^*(f^*(T))\wedge [\Gamma _f]=(\pi _1^*(f^*(T))\wedge [\Gamma _g])^o+\sum _j\lambda _j[V_j].
\end{eqnarray*}
Here $\lambda _j\geq 0$ is a constant, and $V_j$ are varieties of dimension $2$ contained in $\Gamma _f-\Gamma _g$. Moreover, $\pi _2(V_j)$ are contained in the finite set of singular points of $T$.

Since $u$ is smooth, it is not difficult to check that 
\begin{equation}
\pi _1^*(u)(\pi _1^*(f^*(T))\wedge [\Gamma _g])^o=(\pi _1^*(u)\pi _1^*(f^*(T))\wedge [\Gamma _g])^o.
\label{Equation2}\end{equation}
Since we will use similar arguments later on, we give here a detailed proof. Using $T=T^+-T^-$, we may assume that $T$ is positive. We may also assume that $0\geq u\geq -M$. Then $\pi _1^*(u)(\pi _1^*(f^*(T))\wedge [\Gamma _g])^o$ is bounded between the two negative currents $0$ and $-M(\pi _1^*(f^*(T))\wedge [\Gamma _g])^o$. Both these currents have no mass on $\Gamma _f-\Gamma _g$, so is $\pi _1^*(u)(\pi _1^*(f^*(T))\wedge [\Gamma _g])^o$. On $\Gamma _g$, $\pi _1^*(u)(\pi _1^*(f^*(T))\wedge [\Gamma _g])^o$ equals $(\pi _1^*(u)\pi _1^*(f^*(T))\wedge [\Gamma _g])^o$, and the current $(\pi _1^*(u)\pi _1^*(f^*(T))\wedge [\Gamma _g])^o$ has no mass on $\Gamma _f-\Gamma _g$ by definition. Therefore, the two currents in Equation (\ref{Equation2}) are the same on $Y$.   

For any $j$, since $\pi _2(V_j)$ is a point, by the dimension reason we see immediately that 
\begin{eqnarray*}
(\pi _2)_*(\pi _1^*(\varphi )dd^c\pi _1^*(u)\wedge [V_j])=0.
\end{eqnarray*}
Therefore we obtain
\begin{eqnarray*}
(\pi _2)_*(\pi _1^*(\varphi )\wedge dd^c\pi _1^*(uf^*(T))\wedge [\Gamma _f])=(\pi _2)_*(\pi _1^*(\varphi )\wedge dd^c(\pi _1^*(uf^*(T))\wedge [\Gamma _g])^o),
\end{eqnarray*}
and the latter was defined to be $f_*(\varphi dd^cu\wedge f^*(T))$ in Equation (\ref{Equation1}).

2) From Equation (\ref{Equation2}), it suffices to show that 
\begin{eqnarray*}
\lim _{j\rightarrow\infty}(\pi _1^*(u_jf^*(T))\wedge [\Gamma _g])^o=(\pi _1^*(uf^*(T))\wedge [\Gamma _g])^o.
\end{eqnarray*}
The proof of this is similar to that used to prove Equation (\ref{Equation2}). We can assume that $T$ is positive, all $u_j$ and $u$ are negative. Let $R$ be one cluster point of the left hand side. Then $R$ is negative, $R\geq $ the right hand side, and on $\Gamma _g$ then $R=$ the right hand side. Since the right hand side has no mass on $\Gamma _f-\Gamma _g$ by definition, we conclude that $R=$ the right hand side. 
\end{proof}

We write $f^*(T)=\Omega +dd^cu$, where $\Omega $ is a smooth closed $(1,1)$ form, and $u=u^+-u^-$ is a difference of two quasi-psh functions. By Theorem \ref{Theorem1}, the pushforward
\begin{equation}
f_*(\varphi f^*(T)\wedge f^*(T)):=f_*(\varphi \Omega \wedge f^*(T))+f_*(\varphi dd^cu^+\wedge f^*(T))-f_*(\varphi dd^cu^-\wedge f^*(T))
\label{Equation3}\end{equation}
is well-defined. Moreover, if $u_j^{\pm}$ is a sequence of smooth quasi-psh functions decreasing to $u^{\pm}$ then
\begin{equation}
\lim _{j\rightarrow\infty}f_*(\varphi (\Omega +dd^c(u_j^+-u_j^-))\wedge f^*(T))=f_*(\varphi f^*(T)\wedge f^*(T)).
\label{Equation4}\end{equation}
This Bedford-Taylor's monotone convergence type implies the following 
\begin{lemma}
The definition in Equation (\ref{Equation3}) is independent of the choice of $\Omega$ and $u$ in $f^*(T)=\Omega +dd^cu$.
\label{Lemma4}\end{lemma}

\subsection{Definition of the current $T\wedge f_*(\varphi (\Omega +dd^cu)\wedge f^*(T))$}

Let $f_*(\varphi (\Omega +dd^cu)\wedge f^*(T))$ be the current defined in the previous subsection. We now define the intersection $T\wedge f_*(\varphi (\Omega +dd^cu)\wedge f^*(T))$. Without loss of generality we may assume that $0\leq \varphi \leq 1$.

We recall that from Theorem \ref{Theorem1}, if $u=u^+-u^-$ where $u^{\pm}$ are quasi-psh functions, and $u_j^{\pm}$  are smooth quasi-psh functions decreasing to $u^{\pm}$ then
\begin{eqnarray*}
f_*(\varphi (\Omega +dd^cu)\wedge f^*(T))&=&\lim _{j\rightarrow\infty}f_*(\varphi (\Omega +dd^cu_j^+-dd^cu_j^-)\wedge f^*(T))\\
&=&\lim _{j\rightarrow\infty}(\pi _2)_* (\pi _1^*(\varphi ) \pi _1^*(\Omega +dd^cu_j^+-dd^cu_j^-)\wedge \pi _1^*f^*(T)\wedge [\Gamma _f]).
\end{eqnarray*}  

While the sequence $\pi _1^*(\Omega +dd^cu_j^+-dd^cu_j^-)\wedge \pi _1^*f^*(T)\wedge [\Gamma _f]$ may not have a limit, it is a compact sequence and we let $R^+-R^-$  be a cluster point. Here $R^{\pm}$ are positive closed currents of bidimension $(1,1)$ supported in $\Gamma _f$. By the result discussed in the previous paragraph, we have 
\begin{eqnarray*}
f_*(\varphi (\Omega +dd^cu)\wedge f^*(T))=(\pi _2)_*(\pi _1^*(\varphi )R^+-\pi _1^*(\varphi )R^-).
\end{eqnarray*}
Since we assumed that $0\leq \varphi\leq 1$, we have
\begin{eqnarray*}
0\leq (\pi _2)_*(\pi _1^*(\varphi )\wedge R^{\pm})\leq (\pi _2)_*(R^{\pm})
\end{eqnarray*}

\begin{remark}
Note that, under Condition 1 and the assumption that $T$ is in class ($\mathcal{A}$), then $$(\pi _2)_*(R^{+}-R^-)=f_*(f^*(T)\wedge f^*(T))=f_*(f^*(T\wedge T))=T\wedge T$$
has no mass on $I(f^{-1})$. Here we used that $f_*f^*=Id$ on positive closed $(1,1)$ and $(2,2)$ currents, see Theorem 1 in \cite{truong1}. However, each $(\pi _2)_*(R^{\pm})$ may have mass on $I(f^{-1})$. Therefore, if $\varphi$ is not a constant, $(\pi _2)_*(\pi _1^*(\varphi )R^+-\pi _1^*(\varphi )R^-)$ may have mass on $I(f^{-1})$.
\end{remark}

Since the currents $(\pi _2)_*(\pi _1^*(\varphi )\wedge R^{\pm})$ are positive DSH currents in the sense in Dinh-Sibony \cite{dinh-sibony1, dinh-sibony2}, they are $\mathbb{C}$-flat in the sense of Bassanelli \cite{bassanelli}. By Federer-type $\mathbb{C}$-flatness theorem (Theorem 1.24 in \cite{bassanelli}), the restrictions of $(\pi _2)_*(\pi _1^*(\varphi )\wedge R^{\pm})$ to $I(f^{-1})$ are well-defined as a current on $I(f^{-1})$. 

Note that on $Y-I(f^{-1})$, then $f_*(\varphi (\Omega +dd^cu)\wedge f^*(T))=f_*(\varphi )T\wedge T$. Let $(f_*(\varphi )T\wedge T)^o$ be the extension by zero of this current from $Y-I(f^{-1})$ to $Y$. From the discussion above, and taking the bidimension of the various currents into consideration, we obtain the following result

\begin{lemma}
\begin{eqnarray*}
f_*(\varphi (\Omega +dd^cu)\wedge f^*(T))=(f_*(\varphi )T\wedge T)^o+\sum _j(\chi _j^+-\chi _j^-)[C_j].
\end{eqnarray*}
Here $C_j$ are irreducible components of dimension $1$ of $I(f^{-1})$, and $\chi _j^{\pm}$ are bounded positive measurable functions on $C_j$.  

\label{Lemma5}\end{lemma}

Let $A$ be the finite set where $T$ is not smooth. Since $f_*(\varphi )$ is a difference of two quasi-psh functions and $T^{\pm}$ are continuous on $Y-A$, by results in Fornaess-Sibony \cite{fornaess-sibony2} and Demailly \cite{demailly} (Section 4, Chapter 3), the current $f_*(\varphi )T\wedge T$ is well-defined on $Y-A$. Moreover, a monotone convergence property holds. Therefore, since $T^{\pm}\wedge T^{\pm}$ are positive closed currents with no mass on $I(f^{-1})$, an argument similar to that in the proof of Equation (\ref{Equation2}) concludes that $f_*(\varphi )T\wedge T$ has no mass on $I(f^{-1})-A$. By dimension reason, we see that $f_*(\varphi )T\wedge T$ extends as a current on $Y$. The extension current is the same as the current $(f_*(\varphi )T\wedge T)^o$ defined before Lemma \ref{Lemma5}. 

By Lemma \ref{Lemma5}, to define $T\wedge f_*(\varphi (\Omega +dd^cu)\wedge f^*(T))$, it is enough to define $T\wedge (f_*(\varphi )T\wedge T)^o$ and $T\wedge \chi _j^{\pm}[C_j]$ for each $j$. We note that $T\wedge T\wedge T=\mu ^+-\mu ^-$, where $\mu ^{\pm}$ are positive measures which are smooth on $Y-A$. If Condition 2 is satisfied, then we can choose $\mu ^{\pm}$ to have no mass on $A$. Similarly, $T\wedge [C_j]$ is a difference of two positive measures, which we can take to have no Dirac mass if Condition 3 is satisfied. 

The following continuous property is a simple result in measure theory. For completeness, we include a proof of it here.

\begin{lemma} Assume that $T$ is in Class ($\mathcal{A}$) and Condition 2) is satisfied. 

Let $\gamma _n$ be a sequence of uniformly continuous functions on $Y$ which converges to $f_*(\varphi )$ as currents. Moreover, assume that $\gamma _n=f_*(\varphi )$ on an open set $W$ with $W\cap I(f^{-1})=\emptyset$, such that $T$ is smooth on $X-W-I(f^{-1})$. Then the sequence $T\wedge (\gamma _nT\wedge T)^o=\gamma _nT\wedge T\wedge T$ converges to $f_*(\varphi )(\mu ^+-\mu ^-)$. Here the measure $f_*(\varphi )(\mu ^+-\mu ^-)$ is well-defined on $Y-A$, and is defined to be $0$ on the finite set $A$.

A similar result holds when we consider the measures $T\wedge [C_j]$ and the functions $\chi _j^{\pm}$. 
\label{Lemma6}\end{lemma}  
\begin{proof}
Since $\gamma _n$ is smooth, and $T\wedge T$ has no mass on $I(f^{-1})$, we have $(\gamma _nT\wedge T)^o=\gamma _nT\wedge T$. 

Since $\mu ^{\pm}$ are positive smooth measures on $Y-A$, we have 
\begin{eqnarray*}
\lim _{n\rightarrow\infty}\gamma _n T\wedge T\wedge T=f_*(\varphi )(\mu ^+-\mu ^-),
\end{eqnarray*}
on $Y-A$. 

Since $\mu ^{\pm}$ are positive measures with no mass on $A$, any cluster point of $\gamma _n\mu ^{\pm}$, which is bounded by $\mu ^{\pm}$, also has no mass on $A$. Therefore  we obtain 
\begin{eqnarray*}
\lim _{n\rightarrow\infty}\gamma _n T\wedge T\wedge T=f_*(\varphi )(\mu ^+-\mu ^-),
\end{eqnarray*}
on all of $A$.
\end{proof}

By Lemma \ref{Lemma6}, the wedge intersection $T\wedge f_*(\varphi (\Omega +dd^cu)\wedge f^*(T))$ is well-defined using a continuous property. 

\subsection{The case $T$ is smooth or positive}

We now show that in case $T$ is smooth or positive then the Monge-Ampere in our approach Equation (\ref{Equation0}) and the approach in Equation (\ref{Equation?}) are the same. 

We first consider the case where $T$ is smooth. Then, by Theorem \ref{Theorem1}, the Monge-Ampere operator $MA(f^*(T))=f^*(T)\wedge f^*(T)\wedge f^*(T)$ defined in Equation (\ref{Equation0}) is
\begin{eqnarray*}
<MA(f^*(T)),\varphi >=\lim _{j\rightarrow\infty}\varphi f^*(T) \wedge (\Omega +dd^cu_j)\wedge f^*(T),
\end{eqnarray*}
where $u_j$ is an appropriate sequence of smooth functions converging to $u$. Since $T$ satisfies Condition 1, we have $f^*(T)\wedge f^*(T)=f^*(T\wedge T)$. Then 
\begin{eqnarray*}
\lim _{j\rightarrow\infty}\int _X\varphi f^*(T) \wedge (\Omega +dd^cu_j)\wedge f^*(T)&=&\lim _{j\rightarrow\infty}\int _X\varphi (\Omega +dd^cu_j)\wedge f^*(T\wedge T)\\
&=&\lim _{j\rightarrow\infty}\int _Xf_*(\varphi (\Omega +dd^cu_j))\wedge T\wedge T.
\end{eqnarray*}
Since $T$ is smooth on $Y$ and $\Omega +dd^cu_j\rightarrow f^*(T)$, the limit of the sequence of measures $f_*(\varphi (\Omega +dd^cu_j))\wedge T\wedge T$ is exactly $f_*(\varphi f^*(T))\wedge T\wedge T$. It is also the same as $f_*(\varphi )f_*(f^*(T))\wedge T\wedge T$ $=$ $f_*(\varphi )T\wedge T\wedge T$. Here we use that $f_*f^*=Id$ on positive closed $(1,1)$ and $(2,2)$ currents, by Theorem 1 in \cite{truong1}. Thus the proof for the case $T$ is smooth is completed. 

Now we consider the case $T$ is positive. We have the following 

\begin{theorem}
Assume $T$ is a positive closed $(1,1)$ current in Class ($\mathcal{A}$) and satisfies Condition 1. Then

1) \begin{equation}
f_*(\varphi f^*(T)\wedge f^*(T))=f_*(\varphi )T\wedge T.
\label{Equation5}\end{equation}

2) Assume moreover that $T$ satisfies Conditions 2), 3). We write $T=\alpha +dd^cv$, where $\alpha$ is a smooth closed $(1,1)$ form and $v$ is a quasi-psh function. Let $\alpha +dd^cv_n$ be a good approximation of $T$ in the sense of Definition \ref{Definition3}. Then, for any smooth function $\varphi$ on $X$ we have
\begin{eqnarray*}
\lim _{n\rightarrow\infty}\int _Y(\alpha +dd^cv_n)\wedge f_*(\varphi f^*(T)\wedge f^*(T))=\int _YT\wedge f_*(\varphi f^*(T)\wedge f^*(T)).
\end{eqnarray*}

3) Assumptions are as in 2). Assume moreover that $f^*(T\wedge T)$ has no mass on $I(f)$. We write $f^*(T)=\Omega +dd^cu$, where $\Omega$ is a smooth closed $(1,1)$ form and $u$ is a quasi-psh function. Let $\Omega +dd^cu_j$ be a good approximation of $f^*(T)=\Omega +dd^cu$ in the sense of Definition \ref{Definition3}. Then 
\begin{eqnarray*}
\lim _{n\rightarrow\infty}\int _X\varphi f^*(T)\wedge (\Omega +dd^cu_j)\wedge f^*(T)=\int _YT\wedge f_*(\varphi f^*(T)\wedge f^*(T)).
\end{eqnarray*}

\label{Lemma1}\end{theorem}
\begin{proof}
1) In this case, $f^*(T)=\Omega +dd^cu$, where $\Omega $ is a smooth closed $(1,1)$ form and $u$ is a quasi-psh function.    Let $u_j$ be a sequence of smooth quasi-psh functions decreasing to $u$. By the monotone convergence in Equation (\ref{Equation4}), we have 
 \begin{eqnarray*}
 f_*(\varphi f^*(T)\wedge f^*(T))=\lim _{n\rightarrow\infty}f_*(\varphi (\Omega +dd^cu_j)\wedge f^*(T)).
 \end{eqnarray*}

By Theorem 1 in \cite{truong1}, $f_*f^*=Id$ for positive closed $(1,1)$ and $(2,2)$ currents. Since $T$ satisfies Condition 1 and is in Class ($\mathcal{A}$), for every smooth closed $(1,1)$ form $\alpha$ we can apply Theorem 1.1 in \cite{truong} to obtain
\begin{eqnarray*}
f_*(\alpha \wedge f^*(T))=f_*(f^*(f_*\alpha )\wedge f^*(T))=f_*(f^*(f_*(\alpha )\wedge T)=f_*(\alpha )\wedge T.
\end{eqnarray*}

We now claim that 
\begin{eqnarray*}
f_*(\varphi (\Omega +dd^cu_j)\wedge f^*(T))=f_*(\varphi )f_*((\Omega +dd^cu_j)\wedge f^*(T))
\end{eqnarray*}
for every $j$. We choose $\Omega +dd^cu_j$ a good approximation for $f^*(T)$, in the sense of Definition \ref{Definition3}. Therefore $\Omega +dd^cu_j+\epsilon \omega _X$ is positive for every $j$, here $\epsilon >0$ is a constant. Since $\varphi $ is bounded, $f_*(\varphi (\Omega +dd^cu_j+\epsilon \omega _X)\wedge f^*(T))$ is bounded by $f_*((\Omega +dd^cu_j+\epsilon \omega _X)\wedge f^*(T))$. The latter, as seen in the last paragraph, is the same as $f_*(\Omega +dd^cu_j+\epsilon \omega _X)\wedge T$. It has no mass on $I(f^{-1})$. Therefore, $f_*(\varphi (\Omega +dd^cu_j)\wedge f^*(T))$ also has no mass on $I(f^{-1})$. Since $f_*(\varphi (\Omega +dd^cu_j)\wedge f^*(T))=f_*(\varphi )f_*(\Omega +dd^cu_j)\wedge T$  on $Y-I(f^{-1})$, we conclude that the equality holds on all of $Y$.

Recall that $A$ is the finite set of points where $T$ is not smooth. Since $\lim _{j\rightarrow\infty}f_*(\varphi )f_*(\Omega +dd^cu_j)$ $=$ $f_*(\varphi )f_*f^*(T)$ $=$ $f_*(\varphi )T$ on $Y$, we conclude that on $Y-A$
\begin{eqnarray*}
\lim _{j\rightarrow\infty}f_*(\varphi )f_*(\Omega +dd^cu_j)\wedge T=f_*(\varphi )T\wedge T.
\end{eqnarray*}
By dimension reason, the above limit also holds on all of $Y$. 

2) We need to show that
\begin{eqnarray*}
\lim _{n\rightarrow\infty}(\alpha +dd^cv_n)\wedge (f_*(\varphi )T\wedge T)^o=(f_*(\varphi )T\wedge T\wedge T)^o.
\end{eqnarray*}
First, since $\alpha +dd^cv_n$ is smooth, we have
\begin{eqnarray*}
(\alpha +dd^cv_n)\wedge (f_*(\varphi )T\wedge T)^o=(f_*(\varphi )(\alpha +dd^cv_n)\wedge T\wedge T)^o.
\end{eqnarray*}
Therefore, it suffices to show that any cluster point of $(f_*(\varphi )(\alpha +dd^cv_n)\wedge T\wedge T)^o$ has no mass on $I(f^{-1})$.

Since $\alpha +dd^cv_n$ is a good approximation of $\alpha +dd^cv$, there is a constant $\epsilon >0$  such that $\alpha +dd^cv_n+\epsilon \omega _Y$ is positive for every $n$. We write
\begin{eqnarray*}
 (f_*(\varphi )(\alpha +dd^cv_n)\wedge T\wedge T)^o=\mu _{1,n}-\mu _{2},
\end{eqnarray*}
Here 
\begin{eqnarray*}
\mu _{1,n}&=&(f_*(\varphi )(\alpha +dd^cv_n+\epsilon \omega _Y)\wedge T\wedge T)^o,\\
\mu _{2}&=&(f_*(\varphi )\epsilon \omega _Y\wedge T\wedge T)^o,
\end{eqnarray*}
are positive measures.

Since $\mu _{1,n}, \mu _{2}$ are bounded by the positive measures
\begin{eqnarray*}
\nu _{1,n}&=&(\alpha +dd^cv_n+\epsilon \omega _Y)\wedge T\wedge T,\\
\nu _{2}&=&\epsilon \omega _Y\wedge T\wedge T,
\end{eqnarray*}
it suffices to show that $\nu _2$ and any cluster point of $\nu _{1,n}$ have no mass on $I(f^{-1})$.

Since $T$ is smooth outside  a finite number of points and $\omega _Y$ is smooth, it is easy to see that $\nu _2$ has no mass on $I(f^{-1})$.

The limit of $\nu _{1,n}$ is $(T+\epsilon )T\wedge T$ also has no mass on $I(f^{-1})$, since $T\wedge T\wedge T$ has no mass on $I(f^{-1})$ by Condition 3). Here, we use that monotone convergence holds, since $T\wedge T$ is smooth outside a finite number of points.  

3) The proof is similar to the proof of 2). We write $T=\alpha +dd^cv$, where $\alpha$ is a smooth closed $(1,1)$ form, and $v$ is a quasi-psh function. Let $\alpha +dd^cv_n$ be a good approximation of $T$ in the sense of Definition \ref{Definition3}. Hence we can assume that $\alpha +dd^cv_n+\epsilon \omega _Y\geq 0$ for all $n$, here $\epsilon$ is a positive constant. 

Then it is easy to see that
\begin{eqnarray*}
\lim _{j\rightarrow\infty}\int _X\varphi f^*(T)\wedge (\Omega +dd^cu_j)\wedge f^*(T))=\lim _{j\rightarrow\infty}\lim _{n\rightarrow\infty}\int _X\varphi f^*(\alpha +dd^cv_n)\wedge (\Omega +dd^cu_j)\wedge f^*(T).
\end{eqnarray*}

For each $n,j$ then as in 1) and previous sections, we can show that
\begin{eqnarray*}
\int _X\varphi f^*(\alpha +dd^cv_n)\wedge (\Omega +dd^cu_j)\wedge f^*(T)=\int _Y(\alpha +dd^cv_n)\wedge f_*(\varphi (\Omega +dd^cu_j)\wedge f^*(T)). 
\end{eqnarray*} 
 
Therefore, to prove 2), it suffices to show that 
\begin{eqnarray*}
\lim _{j\rightarrow\infty}\lim _{n\rightarrow\infty}(\alpha +dd^cv_n)\wedge f_*(\varphi (\Omega +dd^cu_j)\wedge f^*(T))=(f_*(\varphi )T\wedge T\wedge T)^o.
\end{eqnarray*}
 
Since $f:X-I(f)\rightarrow Y-I(f^{-1})$ is a pseudo-isomorphism, the above equality holds on $Y-I(f^{-1})$. Therefore, we only need to show that any cluster point of  
\begin{eqnarray*}
\lim _{j\rightarrow\infty}\lim _{n\rightarrow\infty}(\alpha +dd^cv_n)\wedge f_*(\varphi (\Omega +dd^cu_j)\wedge f^*(T))
\end{eqnarray*}
has no mass on $I(f^{-1})$. 
 
As in the proof of 1), we have that 
\begin{eqnarray*}
(\alpha +dd^cv_n)\wedge f_*(\varphi (\Omega +dd^cu_j)\wedge f^*(T))=(f_*(\varphi )(\alpha +dd^cv_n)\wedge f_*(\Omega +dd^cu_j)\wedge T)^o.
\end{eqnarray*}
We write 
\begin{eqnarray*}
(f_*(\varphi )(\alpha +dd^cv_n)\wedge f_*(\Omega +dd^cu_j)\wedge T)^o=\mu _{j,n}-\mu _{1,j,n}-\mu _{2,j,n},
\end{eqnarray*}
where 
\begin{eqnarray*}
\mu _{j,n}&=&(f_*(\varphi )(\alpha +dd^cv_n+\epsilon \omega _Y)\wedge f_*(\Omega +dd^cu_j+\epsilon \omega _X)\wedge T)^o,\\
\mu _{1,j,n}&=&(f_*(\varphi )\epsilon \omega _Y\wedge f_*(\Omega +dd^cu_j+\epsilon \omega _X)\wedge T)^o,\\
\mu _{2,j,n}&=&(f_*(\varphi )(\alpha +dd^cv_n+\epsilon \omega _Y)\wedge \epsilon f_*(\omega _X)\wedge T)^o.
\end{eqnarray*}

Note that $\mu _{j,n}$, $\mu _{1,j,n}$, $\mu _{2,j,n}$ are positive measures and are  bounded by the following positive measures 
\begin{eqnarray*}
\nu _{j,n}&=&(\alpha +dd^cv_n+\epsilon \omega _Y)\wedge f_*(\Omega +dd^cu_j+\epsilon \omega _X)\wedge T,\\
\nu _{1,j,n}&=&\epsilon \omega _Y\wedge f_*(\Omega +dd^cu_j+\epsilon \omega _X)\wedge T,\\
\nu _{2,j,n}&=&(\alpha +dd^cv_n+\epsilon \omega _Y)\wedge \epsilon f_*(\omega _X)\wedge T.
\end{eqnarray*} 
 
Hence, it suffices to show that the following limits exist and have no mass on $I(f^{-1})$
\begin{eqnarray*}
&&\lim _{j\rightarrow\infty}\lim _{n\rightarrow\infty}\nu _{j,n},\\
&&\lim _{j\rightarrow\infty}\lim _{n\rightarrow\infty}\nu _{1,j,n},\\
&&\lim _{j\rightarrow\infty}\lim _{n\rightarrow\infty}\nu _{2,j,n}.
\end{eqnarray*}

a) The first limit is 
\begin{eqnarray*}
&&\lim _{j\rightarrow\infty}\lim _{n\rightarrow\infty}(\alpha +dd^cv_n+\epsilon \omega _Y)\wedge f_*(\Omega +dd^cu_j+\epsilon \omega _X)\wedge T\\&=&\lim _{j\rightarrow\infty}f_*(\Omega +dd^cu_j+\epsilon \omega _X)\wedge (T+\epsilon \omega _Y)\wedge T\\
&=&(T+f_*(\epsilon \omega _X))\wedge (T+\epsilon \omega _Y)\wedge T.
\end{eqnarray*}
Here we used that $T$ is smooth outside a finite number of points, hence monotone convergence holds. In the resulting limit: 

- The term $T\wedge T\wedge T$ has no mass on $I(f^{-1})$ by Condition 3). 

- The term $T\wedge \omega _Y\wedge T$ has no mass on $I(f^{-1})$ since $T$ is smooth outside a point and $\omega _Y$ is smooth.

- The term $f_*(\omega _X)\wedge \omega _Y\wedge T$ has no mass on $I(f^{-1})$  since $T$ is smooth outside a finite number of points, $f_*(\omega _X)$ has no mass on proper analytic subvarieties, and $\omega _Y$ is smooth. 

- Now we show that the last term $f_*(\omega )\wedge T\wedge T$ has no mass on $I(f^{-1})$. By assumption, $f^*(T\wedge T) $ has no mass on $I(f)$, hence it is a positive current, and the positive measure $\omega _X\wedge f^*(T\wedge T)$ has no mass on $I(f)$. Therefore, the pushforward $f_*(\omega _X\wedge f^*(T\wedge T))$ is well-defined as a positive measure with no mass on $I(f^{-1})$. On $Y-I(f^{-1})$, then $f_*(\omega )\wedge T\wedge T=f_*(\omega _X\wedge f^*(T\wedge T))$. Therefore, $f_*(\omega )\wedge T\wedge T\geq f_*(\omega _X\wedge f^*(T\wedge T))$ on $Y$. Moreover, the masses of the two measures $f_*(\omega )\wedge T\wedge T$ and $f_*(\omega _X\wedge f^*(T\wedge T))$, which can  be computed cohomologously, are the same. We conclude that $f_*(\omega )\wedge T\wedge T= f_*(\omega _X\wedge f^*(T\wedge T))$ on $Y$. Here we use the following properties of pseudo-isomorphisms in dimension $3$: $f^*(\zeta ).f^*(\eta )=f^*(\zeta .\eta )$ (see \cite{bedford-kim}) for $\zeta \in H^{1,1}$ and $\eta \in H^{2,2}$. 

Hence we conclude that the first limit has no mass on $I(f^{-1})$.

b) Using a similar argument, we have that the second and third limits also have no mass on $I(f^{-1})$, as wanted.
\end{proof}

\end{document}